\documentclass[12pt]{amsart}
\usepackage{amsmath,amssymb,amsthm,amsfonts,amscd}
\usepackage{microtype}
\usepackage{mathscinet}
\usepackage[pdftex]{color}
\usepackage[bookmarks=true,hyperindex,pdftex,colorlinks, citecolor=MidnightBlue,linkcolor=MidnightBlue, urlcolor=MidnightBlue]{hyperref}
\usepackage[dvipsnames]{xcolor}

\newcommand{\diam}{\operatorname{diam}}

\newcommand{\ddelta}{\overline{\delta}}
\newcommand{\vvartheta}{\overline{\vartheta}}
\newcommand{\NN}{\mathfrak N}
\newcommand{\QQ}{\mathfrak Q}
\newcommand{\SE}{\mathfrak S}
\newcommand{\Dz}{\operatorname{Dz}}
\newcommand{\Cz}{\operatorname{Cz}}
\newcommand{\Sz}{\operatorname{Sz}}
\newcommand{\eqnorm}[1]{{\left\vert\kern-0.25ex\left\vert\kern-0.25ex\left\vert #1 
    \right\vert\kern-0.25ex\right\vert\kern-0.25ex\right\vert}}
\newcommand{\norm}[1]{\left\lVert#1\right\rVert}
\def\norma{\norm{\hspace*{.35em}}}
\newcommand\eqnorma{\eqnorm{\hspace*{.35em}}}

\newcommand{\conv}{\operatorname{conv}}
\newcommand{\cconv}{\overline{\conv}}

\newcommand{\M}{\mathcal M}

\begin{document}

\baselineskip=14pt

\newtheorem{thm}{Theorem}[section]
\newtheorem{theo}[thm]{Theorem}
\newtheorem{prop}[thm]{Proposition}
\newtheorem{coro}[thm]{Corollary}
\newtheorem{lema}[thm]{Lemma}
\newtheorem{defi}[thm]{Definition}
\newtheorem{ejem}[thm]{Example}
\newtheorem{rema}[thm]{Remark}
\newtheorem{fact}[thm]{Fact}
\newtheorem{open}[thm]{PROBLEM}

\title{Uniformly convex renormings and \linebreak generalized cotypes}

\author{Luis C. Garc\'ia-Lirola}
\address {Department of Mathematical Sciences, Kent State University, Kent, Ohio 44242, USA}
\email {lgarcial@kent.edu} 

\author{Mat\'ias Raja}
\address{Departmento de Matem\'aticas, Universidad de Murcia, 30100 Espinardo, Murcia, Spain}
\email{matias@um.es}

\thanks{This work was supported by the Grants of Ministerio de Econom\'ia, Industria y Competitividad MTM2017-83262-C2-2-P; and Fundaci\'on S\'eneca Regi\'on de Murcia 20906/PI/18. The first author is also supported by a postdoctoral grant from Fundaci\'on S\'eneca.}

%% 46B20 Geometry and structure of normed linear spaces
%% 46B03 Isomorphic theory (including renorming) of Banach spaces
\date {November 22, 2019} 

\begin{abstract}
We are concerned about improvements of the modulus of convexity by renormings of a super-reflexive Banach space. Typically optimal results are beyond Pisier's power functions bounds $t^p$, with $p \geq 2$, and they are related to the notion of generalized cotype. We obtain an explicit upper bound for all the modulus of convexity of equivalent renormings and we show that if this bound is equivalent to $t^2$, the best possible, then the space admits a renorming with modulus of power type $2$. We show that a UMD space admits a renormings with modulus of convexity bigger, up to a multiplicative constant, than its cotype. We also prove the super-multiplicativity of the supremum of the set of cotypes.
\end{abstract}

\maketitle

\section{Introduction}

The classical Lebesgue spaces $L^p$ play a central role in Banach space theory, not only as examples but also in many chapters as bases, operators, linear and nonlinear classification, for instance. The very definition of these spaces leans on the convexity of the power function $t^p$ with $p\geq 1$. There are at least two other topics in Banach space geometry where power functions $t^p$ arise in such a way that the best exponent $p$ that could be placed there provides valuable information (isomorphic and nonlinear) about the structure of the space. One of those topics is the theory of type and cotype (always Rademacher along this paper), well known in the linear theory but with important applications and generalizations in the nonlinear settings. The other topic affects only to super-reflexive spaces, and consists in the functions that can be modulus of uniform smoothness or convexity of equivalent norms. After Pisier's famous result power functions, up to a constant, can do the work. Both topics are not totally independent. For instance, if $X$ has an equivalent norm with a modulus of uniform convexity bonded below by $ct^p$ with $c>0$ and $p \geq 2$ then $p$ is a cotype of $X$.\\

The spaces $L^p$ are considerably generalized by Orlicz spaces $L^\Phi$ as one pass from the linear order of the parameter $p \in [1,+\infty)$ to the much more complex order among the Orlicz functions defined on 
${\mathbb R^{+1}}$. It turns out that the use of Orlicz-Young functions $\Phi$ instead of the power functions $t^p$ allows us to obtain finer information on the space $X$ when applied to the topics discussed above. According to Figiel \cite{figiel2}, a nonnegative nondecreasing function $\phi(t)$ is a {\it generalized cotype} of the space $X$
if there exist constants $a,b>0$ such that $\sum_{k=1}^{n} \phi(\|x_k\|)
 \leq b$ whenever $x_1,\dots,x_n \in X$ satisfy
$$ \int_{0}^{1} \| \sum_{k=1}^{n} r_k(t) \, x_k \| \, dt \leq a $$
where $(r_k(t))$ are the Rademacher functions, see \cite{banach, LT} for instance.  A remarkable
result of Figiel and Pisier \cite{FP} establishes that any modulus of convexity is a generalized cotype. There is also a notion of {\it generalized type}, not only for spaces but also for operators, and a duality theory that we shall no consider here. See \cite{Mas} for more information on this subject.\\

In relation to moduli of convexity under equivalent renormings,  the second named author proved in \cite{raja2} the following result.

\begin{theo}[\cite{raja2} Theorem 1.2]\label{main1}
Let $X$ be a superreflexive Banach space. There exists a positive decreasing
submultiplicative function ${\mathfrak N}_X(t)$
defined on $(0,1]$ satisfying that ${\mathfrak N}_X(t)^{-1}$ is 
the supremum, up to equivalence,
with respect to the order $\preceq$ of the set
$$ \{ \delta_{|\!|\!| \cdot |\!|\!|}(t): |\!|\!| \cdot |\!|\!| ~\mbox{is an equivalent norm on}~ X \}. $$
\end{theo}
\vspace{2mm}

The order $\preceq$ is established for functions defined on $(0,1]$. 
We write $\phi \preceq \psi$ if there is
a constant $c>0$ such that $\phi(t) \leq c \, \psi(t)$ for all $t \in (0,1]$. If
$\phi \preceq \psi$ and $\psi \preceq \phi$, then we say that $\phi$ and 
$\psi$ are {\it equivalent}, denoted by $\sim$. A Baire category argument shows that we may use pointwise order if we restrict ourselves to norms within an interval of equivalence with respect to a fixed norm, for instance, $2$-equivalent norms. Recall that the norms $\norma_1$ and $\norma_2$
are $\gamma$-{\it equivalent} for $\gamma>1$
if $\gamma^{-1} \|x\|_1 \leq \norm{x}_2 \leq \gamma \|x\|_1$.\\

The function ${\mathfrak N}_X(t)^{-1}$ can be showed to be equivalent to an Orlicz-Young function defined on $[0,\infty)$ thanks to the submultiplicativity of ${\mathfrak N}_X(t)$. The construction of ${\mathfrak N}_X(t)$  given in \cite{raja2} is rather complicated. The first aim of this paper is to provide a different construction that casts additional features, in particular, an explicit form.

\begin{theo}\label{main-Szlenk}
For any super-reflexive space $X$ we have ${\mathfrak N}_X(t) \sim \Sz(B_{L^2(X)},t)$.
\end{theo}

Here $\Sz$ stands for the Szlenk index whose definition is reminded later. The proof of this theorem relies on the fine results of G.~Godefroy, N.~Kalton and G.~Lancien \cite{GKL} and it is totally independent of the arguments from \cite{raja2}, so it provides new connections between uniformly convex renorming and {\it asymptotically uniformly convex} renorming.\\

However, the methods we have used do not provide an answer to the problem of obtaining a renorming of a Banach space with modulus of uniform convexity greater than $c \, \phi(t)$ where $c>0$ and $\phi(t)$ is a given, and feasible, positive function. Such a function must be a cotype of $X$, so the question could be reformulated as if given a cotype $\phi(t)$ on a Banach space we can renorm it with a modulus of convexity better than $\phi(t)$. That problem was investigated by Figiel \cite{figiel, figiel1, figiel2} and solved positively in Banach spaces with unconditional bases (and more generally l.u.st. spaces). In absence of bases, martingales seem to us a suitable tool since they can manage information coming from arbitrarily separated vectors. In order to better understand this, note that the information provided by the dentability index just tells  us about the growing of uniformly separated martingales (bushes). Renormings whose modulus of convexity is a generalized cotype using martingales were studied by Garling \cite{Garling} who provided a characterization, however no new example was included.\\  

A Banach space $X$ is said to have the {\it uncondicional martingale difference property}  (UMD for short) if for some $1<p<+\infty$ (equivalently, for all $1<p<+\infty$) any $X$ valued martingale $(f_n)$ which is bounded in $L^p(X)$ has unconditionally convergent differences, that is, $\sum_{n=1}^\infty df_n$ is unconditionally convergent in $L^p(X)$ , where $df_n=f_n-f_{n-1}$. Notably UMD spaces are super-reflexive and they include many classical spaces as $L^p$ for $1<p<+\infty$ or the reflexive Orlicz spaces \cite{FerGar}. We have this generalization of the known fact that an UMD Banach space with classic cotype $p$ can be renormed with modulus of convexity bounded below by $c \, t^p$ \cite[Proposition~10.40]{Pisier_libro}.

\begin{theo}\label{UMD-renorm}
Let $X$ be an UMD Banach space and $\phi(t)$ a generalized cotype on it. Then there exists an equivalent norm 
$\eqnorma$ on $X$ such that $\delta_\eqnorma (t) \geq c \, \phi(t)$ for some $c>0$.
\end{theo}

Note that in UMD spaces, after the previous Theorem, or in super-reflexive spaces with unconditional bases, after Figiel's results, the function $\NN_X(t)^{-1}$ becomes the supremum of the set of cotypes with respect to the order used in Theorem~\ref{main1}. So as to discuss the supremum of the cotypes independently of uniformly convex renorming let us introduce the following definition. We say that a cotype $\phi$ is {\it normalized} if $\sum_{k=1}^n \phi(\|x_k\|) \leq 1$ whenever 
$$ \int_{0}^{1} \| \sum_{k=1}^{n} r_k(t) \, x_k \| \, dt \leq 1. $$
We will always assume that a normalized cotype is defined just on $[0,1]$. With respect to the supremum of the cotypes, we have the following result in the spirit of Theorem~\ref{main1}.

\begin{theo}\label{sup-cotype}
Let $X$ be a superreflexive Banach space. There exists a positive decreasing
submultiplicative function ${\mathfrak Q}_X(t)$
defined on $(0,1]$ satisfying that
$$ \sup \{ \phi(t): \phi \mbox{ is a convex normalized cotype of } X \} \leq ( {\mathfrak Q}_X(t) - 1)^{-1} $$
 and for every $\varepsilon>0$ the function  $\phi_\varepsilon(t)=\QQ_X(\varepsilon)^{-1}$ if $t \geq \varepsilon$ and  $\phi_\varepsilon(t)=0$ otherwise is a normalized cotype. Moreover, $\QQ_X(t)^{-1}$ is the supremum with respect to $\preceq$ of the set of all of cotypes of $X$ defined on $[0,1]$.
\end{theo}

An easy application of Dvoretzky's theorem implies that $\QQ_X(t)^{-1} \preceq t^2$.
If there is a best cotype it has to be equivalent to ${\mathfrak Q}_X(t) ^{-1} $, however the existence of a best cotype is not guaranteed. Figiel \cite{figiel3} found an example of Banach space with no best cotype nor  best modulus of convexity. Nevertheless, there are important classes of Banach spaces where there are a best cotype and a best modulus of convexity, like the reflexive Orlicz spaces. With aims in a more general result, we have proved that if the function $\NN_X(t)$ is asymptotically the smallest possible then there is a renorming with a best modulus of convexity. 

\begin{theo}\label{dos}
Let $X$ be a super-reflexive Banach space such that $\NN_X(t) \sim t^{-2}$. Then $X$ has an equivalent norm $\eqnorma{}$ such that $\delta_\eqnorma{} (t) \geq c \, t^2$ for some $c>0$. As a consequence $X$ has a renorming with the best modulus of uniform convexity and its best cotype is $t^2$.
\end{theo}

The paper is organized as follows. After this introduction, the second section contains some material on asymptotical uniformly convex renormings that we will need for the proof of Theorem~\ref{main-Szlenk} done in the third section. The fourth section is devoted to the implications for the Bochner-Lebesgues spaces $L^p(X)$ of a generalized cotype on $X$. That will allow us to prove the result on UMD spaces. The last section deals with the suprema of cotypes and moduli of convexity.\\

The authors wish to express their thanks to S.~Troyanski for drawing our attention to Figiel's almost unnoticed work on generalized cotypes and uniformly convex renorming.

\section{Asymptotic uniformly convexity revisited}

In order not to deal with the version for duals of the moduli all the Banach spaces are supposed to be reflexive in this section. 
Recall that the \emph{modulus of uniform convexity} of $(X,\norma)$ is given by
\[ \delta_{\norma}(t) = 1- \frac{1}{2}\sup\{\norm{x+y}:
\norm{x} = \norm{y} =1, \norm{x-y} > t \} \]
and the \emph{modulus of asymptotic uniform convexity} is given by
\[ \ddelta_{\norma}(t) = \inf_{\norm{x}=1}\sup_{\dim(X/Y)<\infty}\inf_{y\in Y, \norm{y}=1} \norm{x+ty}-1 \]
where $t\in (0,1]$. A Banach space $(X, \norma)$ is \emph{uniformly convex} (UC) if $\delta_{\norma}(t)>0$ for every $t$, and it is said to be \emph{asymptotically uniformly convex} (AUC) if $\ddelta_{\norma}(t)>0$ for every $t$.  The good sequential properties of the weak topology in reflexive spaces implies that we have
\[\ddelta_{\norma}(t)=\inf\{\liminf_n\norm{x+x_n}-1: \norm{x}=1, \norm{x_n}=t, x_n\stackrel{w}{\to} 0\}.\]

We wish to simplify computations when comparing equivalent norms on the same Banach space, so instead of using the usual modulus of convexity for an equivalent norm
$\eqnorma$, we will consider the following
relative moduli  %(relative to $\| \cdot \|$)
\begin{align*}
\vartheta_{\eqnorma}(t) &= 1- \frac{1}{2}\sup\{\eqnorm{x+y}:
\eqnorm{x} = \eqnorm{y} =1, \norm{x-y} > t \},\\
\vvartheta_{\eqnorma}(t) &= 1- \sup\{\eqnorm{x} : \exists (x_n)_{n=1}^\infty\stackrel{w}{\to} x, \eqnorm{x_n}\leq 1, \inf_{n\neq m} \norm{x_n-x_m}>t\}
\end{align*}
with the convection that $\sup\emptyset = 0$. 
Clearly we have \[\delta_{\eqnorma}(\gamma^{-1}t)
\leq \vartheta_{\eqnorma}(t)\leq \delta_{\eqnorma}(\gamma t)\]
if $\eqnorma$
is $\gamma$-equivalent to $\norma$ and $0<t\leq \gamma^{-1}$. Moreover, the same arguments that provide the equivalence between uniformly Kadec-Klee and asymptotic uniform convexity in reflexive spaces relate $\ddelta_{\eqnorma}$ and $\vvartheta_{\eqnorma}$. %Let us include a proof for completeness. 

\begin{prop} Let $(X,\norma)$ be a reflexive space and $\eqnorma$ be a norm $\gamma$-equivalent to $\norma$. Then 
\[ \ddelta_{\eqnorma}(2^{-1}\gamma^{-1}t)
\leq \vvartheta_{\eqnorma}(t)\leq \ddelta_{\eqnorma}(2\gamma t) \quad \text{ for all } 0<t<(2\gamma)^{-1}.\]
\end{prop}
\begin{proof} Let $(x_n)$ be a sequence in $B_{\eqnorma}$ weakly-convergent to $x$ with $\norm{x_n-x_m}>t$ when $n\neq m$. By extracting a subsequence, we may assume that $\norm{x_n-x}>t/2$. Fix $\varepsilon>0$ and let $Y$ be a finite-codimensional subspace of $X$ such that 
\[ \eqnorm{\frac{x}{\eqnorm{x}} + y }-1 \geq \ddelta_\eqnorma(2^{-1}\gamma^{-1}t) - \varepsilon \]
whenever $y\in Y$, $\eqnorm{y}\geq 2^{-1}\gamma^{-1}t$. Since $x_n-x\stackrel{w}{\to}0$, we can find $y_n\in Y$ such that $\norm{x_n-x-y_n}\to 0$. Clearly $\eqnorm{y_n}\geq 2^{-1}\gamma^{-1}t$ eventually. 
Thus 
\begin{align*}
\ddelta_\eqnorma(2^{-1}\gamma^{-1}t) - \varepsilon &\leq \liminf_{n\to\infty} \eqnorm{\frac{x}{\eqnorm{x}} + y_n} -1 \\
& = \liminf_{n\to\infty} \eqnorm{\frac{x}{\eqnorm{x}} + x_n-x} - 1\leq 1-\eqnorm{x}.
\end{align*}
It follows that $\ddelta_{\eqnorma}(2^{-1}\gamma^{-1}t)\leq \vvartheta_{\eqnorma}(t)$. 

Now, assume that $ \vvartheta_{\eqnorma}(t)>\alpha> \ddelta_{\eqnorma}(2\gamma t)$. Then there is $x\in X$ with $\eqnorm{x}=1$ such that $\inf_{y\in Y,\eqnorm{y}=2\gamma t} \eqnorm{x+y}<1+\alpha$ for every finite-codimensional space $Y$. Take $x^*\in X^*$ with $\eqnorm{x^*}=1$ and $x^*(x)=1$. One can easily construct by induction a sequence $(y_n)_{n=1}^\infty\subset \ker x^*$ such that $\eqnorm{y_n}= 2\gamma t$, $\eqnorm{y_n-y_m}\geq 2\gamma t$ if $n\neq m$ and $\eqnorm{x+y_n}<1+\alpha$. By reflexivity, we may assume that $y_n\stackrel{w}{\to} y$. Notice that 
\[\eqnorm{x+y}\geq x^*(x+y) = \lim_{n\to\infty} x^*(x+y_n)= 1.\]
On the other hand,
\[ \norm{\frac{x+y_n}{1+\alpha}-\frac{x+y_m}{1+\alpha}}\geq \frac{2t}{1+\alpha}> t\]
since $\alpha<\vvartheta_{\eqnorma}(t)\leq 1$, and so 
\[ \frac{1}{1+\alpha}\leq\eqnorm{\frac{x+y}{1+\alpha}} \leq 1-\vvartheta_{\eqnorma}(t).\]
Thus,
\[1\leq (1+\alpha)(1-\vartheta_{\eqnorma}(t)).\]Letting $\alpha$ go to $\vvartheta_{\eqnorma}(t)$ we get a contradiction. 
\end{proof}

The most important result about AUC renorming was proved in \cite{GKL}. For the statement we need the convex Szlenk index $\mbox{Cz}(B_X, t)$ that will be explained in detail in the next section. We just need to know now that $\mbox{Cz}(B_X, t)$ is a decreasing function defined for $t \in (0,1]$ and taking values in $[1,+\infty]$.

\begin{theo}[\cite{GKL} Theorem 4.7]\label{GLK-th}
Let $X$ be a separable reflexive Banach space with $\mbox{Cz}(B_X, 1) < +\infty$. Then there exists $1<C<19200$ such that for every $0<\tau<1$ there is a 2-equivalent norm $|\!|\!|.|\!|\!|_{\tau}$ on $X$ such that 
$$ \ddelta_{|\!|\!|.|\!|\!|_{\tau}}  (\tau) \geq \mbox{Cz}(B_X, \tau/C)^{-1} .$$  
\end{theo}

Since we are working without separability assumptions it would be desirable to have a nonseparable version of the previous result with the sharp estimation. That was achieved by Causey in \cite{causey}. However, we will provide an alternative proof showing that the general case can be reduced to the reflexive one with a simple ``gluing argument'' for AUC norms. Notice the similarities with the gluing argument for uniformly convex renormings of finite-dimensional subspaces \cite[Proposition~9.2]{banach}. Let $\SE(X)$ denote the lattice of separable subspaces of $X$.

\begin{prop}\label{glue}
Let $X$ be a Banach space and $\beta(t)\geq 0$ a function defined for $t \geq 0$ such that for all $S \in \SE(X) $ there is norm $\norma_S$ defined on $S$ which is 2-equivalent there to the original norm and satisfies that $\vvartheta_{\norma_S}(t) \geq \beta(t)$. Then there exists a 2-equivalent norm $\eqnorma$ on $X$ such that $\vvartheta_{\eqnorma}(t) \geq \beta(t/\lambda)$ for every $\lambda>1$.
\end{prop}

\begin{proof}
Extend the function $\norma_S$ given by the hypothesis to all $X$ by setting $\| x \|_S=0$ if $x\not \in S$. Define now
$$ \eqnorm{x}_S = \sup\{ \|x\|_R: R \in \SE(X), S \subset  R\}. $$
It is clear that $\eqnorma_S$ is a 2-equivalent norm on $X$ for every $S \in \SE(X)$. Observe that the net $(\eqnorm{x}_S)_{S \in \SE(X)}$ converges to $\inf\{\eqnorm{x}_S: S \in \SE(X) \}$. Its limit defines a 2-equivalent norm on $X$
$$ \eqnorm{x}= \lim_{S \in \SE(X)}  \eqnorm{x}_S .$$
Fix $\lambda>1$ and $\varepsilon >0 $ such that $1+\varepsilon <\lambda$. Assume $t>0$ and take $(x_n) \subset B_{\eqnorma}$ a $t$-separated sequence which weakly converges to $x$. By the very definition of $\eqnorma$ we may take $S \in \SE(X)$ such that $ \eqnorm{x_n}_S \leq 1+\varepsilon$ for all $n \in {\mathbb N}$. Indeed, such separable subspace exists for every $x_n$ and then we can take the subspace generated by the union. Therefore $\|x_n\|_R \leq 1+\varepsilon$ for all $n \in {\mathbb N}$ and $R \in \SE(X)$ with $S \subset R$. Note that $(x_n/(1+\varepsilon))$ is a $t/\lambda$-separated sequence in $R$ and thus its limit satisfies that 
$$ \left\| \frac{x}{1+\varepsilon} \right\|_R \leq 1 - \beta(t/\lambda) $$ 
and so
$$ \eqnorm{ \frac{x}{1+\varepsilon} } \leq \eqnorm{ \frac{x}{1+\varepsilon} }_S \leq 1 - \beta(t/\lambda)  .$$
As $\varepsilon>0$ was arbitrary we deduce that $\vvartheta_{ \eqnorma}(t) \geq \beta(t/\lambda)$.
\end{proof}

Now we can establish the non-separable reflexive version of the result of Godefroy, Kalton and Lancien.

\begin{coro}\label{nonsepKGL}
Let $X$ be a reflexive Banach space such that $ \mbox{Cz}(B_X, 1) < +\infty$. Then there exists $1<C<19201$ such that for every $0<\tau<1$ there is a 2-equivalent norm $\eqnorma_{\tau}$ on $X$ such that 
$$ \vvartheta_{\eqnorma_{\tau}}  (\tau) \geq \mbox{Cz}(B_X, \tau/C)^{-1} .$$ 
\end{coro}

\begin{proof} 
Take $\beta(t) = 0$ if $t < \tau$ and $\beta(t)=\mbox{Cz}(B_X, \tau/C')^{-1}$ otherwise with $C'$ the constant given by Theorem~\ref{GLK-th}. Then set $\lambda=19201/19200$ and set $C=\lambda C'$.
\end{proof}

According to the results from \cite{raja2} we have
\[ {\mathfrak N}_X(t) \sim
( \sup \{ \vartheta_{\eqnorma}(t): \eqnorma \text{ is}~2\text{-equivalent to}~ \norma \} )^{-1}.\]
Actually we will use the equivalence above as a definition for the original construction of ${\mathfrak N}_X(t)$
is quite technical and leads to a property that we will not use here. However, we will need the ``asymptotic'' version of ${\mathfrak N}_X(t)$ whose follows the formula above.

\begin{defi} Let $X$ be a Banach space admiting an equivalent AUC norm. We define
\[ \overline{{\mathfrak N}}_X(t) = 
( \sup \{ \vvartheta_{\eqnorma}(t): \eqnorma \text{ is}~2\text{-equivalent to}~ \norma )^{-1}.\]
\end{defi}

\begin{coro}\label{NCZ}
Let $X$ be a reflexive space such that $\Sz(B_X) \leq \omega$, then there exists a constant $c>0$ such that 
\[ \overline{{\mathfrak N}}_X(t) \leq \Cz(B_X, c \, t) \]
for every $t\in (0,c]$.
\end{coro}

\begin{proof}
There exists a universal constant $C>0$ such that for every $0<t<1$ there is an $2$-equivalent norm $\eqnorma_t$ such that $\overline{\vartheta}_{\eqnorma_t}(t)\geq \Cz(B_X,  t/C)^{-1}$ following that $\overline{{\mathfrak N}}_X(t) \leq \overline{\vartheta}_{\eqnorma_t}(t)^{-1} \leq \Cz(B_X, t/C)$.
\end{proof}

 %such that \[ 1-\overline{\vartheta}_{\| \cdot \|_t}(t) \geq \frac{1}{1+\Cz(B_X,ct)^{-1}} \geq  1 - \Cz(B_X,ct)^{-1} \] following that $\overline{{\mathfrak N}}_X(t)^{-1} \leq \overline{\vartheta}_{\| \cdot \|_t}(t) \leq \ddelta_{\|.\|_t}(4t) \leq Cz(B_X, c \, t)^{-1}$.

\section{Szlenk indices and renormings of $L^2(X)$}

We will start this section by reminding the ordinal indices associated to renorming of Banach spaces. Again we will assume that $X$ is a reflexive Banach space. In particular, we will give the definition of the function $\Cz(B_X,t)$ used in some statements of the previous section. We will consider the following set derivations:
\begin{align*} s'_t(A) &= \{ x\in A :  \forall U\subset X \text{ weak open } (x\in  U \Rightarrow \diam(A \cap U) \geq t)\},\\
c'_t(A) &= \cconv(s'_t(A)),\\
d'_t(A) &= \{ x\in A : \forall H\subset X \text{ open halfspace } (x\in  H \Rightarrow \diam(A \cap H) \geq t)\}
\end{align*}
where $t > 0$ and $A$ is a bounded subset of $X$. The names are {\it Szlenk}, {\it convex Szlenk} and {\it dentability} derivations respectively.
Clearly, the set derivations defined above are monotone. Thus, they can be iterated:
\[ s^{n+1}_t(A) = s^{'}_t(s^n_t(A)), \quad c^{n+1}_t(A) = c^{'}_t(c^n_t(A)),\quad d^{n+1}_t(A) = d^{'}_t(d^n_t(A))\]
We define 
\begin{align*}
\Sz(B_X,t)&=\min\left\{n: s^n_t(B_X) = \emptyset\right\},\\ \Cz(B_X,t)&=\min\left\{n: c^n_t(B_X) = \emptyset\right\},\\ \Dz(B_X,t)&=\min\left\{n: d^n_t(B_X) = \emptyset\right\}.
\end{align*}
These quantities are well defined and finite when $X$ is super-reflexive. Indeed, this follows from the existence of a uniformly convex renorming of $X$ and the following well-known facts:
\begin{itemize}
	\item[(i)] If $\eqnorma$ is $\gamma$-equivalent to the norm of $X$ then
	\[ s^n_{\gamma^2 t}(B_X) \subset \gamma s^n_{t}(B_\eqnorma) \]
	for every $n\in \mathbb N$, and so 
	\[ \Sz(B_X,\gamma^2 t)\leq \Sz(B_\eqnorma, t) , \]
	and analogous statements for the convex Szlenk and the dentability derivations and indices hold.
	\item[(ii)] $\Sz(B_X,t)\leq \Cz(B_X,t ) \leq \Dz(B_X,t)$. 
	\item[(iii)] $ \Dz(B_X,t)\leq \delta_X(t)^{-1} +1$; $\Cz(B_X,t ) \leq \ddelta_X(t/2)^{-1} +1.$  
\end{itemize}

The factor $1/2$ in the last estimation could be skipped with an alternative and ``more topological'' definition of the modulus $\ddelta$. Note also that $\Sz(B_X, 1) < +\infty$ and $\Cz(B_X, 1) < +\infty$ are equivalent statements, whereas $\Dz(B_X, 1) < +\infty$ is equivalent to super-reflexivity.\\

The next result combined with Corollary~\ref{NCZ} shows that $\overline{{\mathfrak N}}_X(t)$ is actually equivalent to the convex Szlenk index.

\begin{coro}\label{th:CzN}
Let $X$ be a reflexive space such that $\Sz(B_X, 1) < +\infty$, then \[ \Cz(B_X, 32t) \leq 2 \, \overline{{\mathfrak N}}_X(t). \]
\end{coro}

\begin{proof}
For any  $2$-equivalent $\eqnorma$ and every $0<t<1$ we have
\begin{align*}
2^{-1}\Cz(B_X, 4t)\leq \Cz(B_{\eqnorma},t)-1 \leq \ddelta_{\eqnorma}(t/2)^{-1} \leq \vvartheta_{\eqnorma}(t/8)^{-1}.
\end{align*}
The proof finishes taking $2$-equivalent norms  $\eqnorma$ such that $\vvartheta_{\eqnorma}(t/8)^{-1}$ approaches $\overline{{\mathfrak N}}_X(t/8)$ and  the obvious change of the variable.
\end{proof}

\begin{rema}
We could show that $\overline{{\mathfrak N}}_X(t)$ is equivalent to a submultiplicative function, as ${\mathfrak N}_X(t)$ is. Therefore $\overline{{\mathfrak N}}_X(t) \sim \Cz(B_X,t)$. We skip the proof as we will apply these results in spaces $L^2(X)$ where the submultiplicativity is given in an easy way.
\end{rema}

It was shown in \cite{GKL} that $\Sz(B_X, t)$ is equivalent to $\Cz(B_X,t)$ if $X$ is super-reflexive. We only need the following particular case of that result. 

\begin{prop}\label{prop:eqL2}
If $X$ is super-reflexive, then the functions $\Sz(B_{L^2(X)},t)$,
$\Cz(B_{L^2(X)},t)$ and $\Dz(B_{L^2(X)},t)$ are equivalent.
\end{prop}

\begin{proof} Clearly it suffices to show that $\Dz(B_{L^2(X)},t)\preceq \Sz(B_{L^2(X)},t)$.
%Clearly the following holds
%\[ \Sz(B_{L^2(X)},t) \leq \Cz(B_{L^2(X)},t) \leq \Dz(B_{L^2(X)},t).\]
For any Banach space $Y$ we have $\Dz(B_Y,t) \leq \Sz(B_{L^2(Y)},t/2)$, see \cite{Lancien3, Lancien2}. In particular, we have
\[ \Dz(B_{L^2(X)},t) \leq \Sz(B_{L^2(L^2(X))},t/2). \]
As $L^2(L^2(X))$ is isometric to $L^2(X)$, we have 
\[ \Dz(B_{L^2(X)},t) \leq \Sz(B_{L^2(X)},t/2) \leq c \Sz(B_{L^2(X)},t) \]
where $c=Sz(B_{L^2(X)},1/2)$ by the submultiplicativity of the Szlenk index.
\end{proof}

\begin{coro}
If $X$ is a super-reflexive Banach space then $\overline{{\mathfrak N}}_{L^2(X)}(t)$ is equivalent to $\Sz(B_{L^2(X)}, t)$ and so it is equivalent to a submultiplicative function.
\end{coro}
\begin{proof}
Combine the previous proposition with Corollary~\ref{NCZ}, Corollary~\ref{th:CzN} and the fact that $L^2(X)$ is super-reflexive whenever $X$ is, see \cite[Proposition~11.39]{Pisier_libro} for instance. 
\end{proof}

\begin{theo}\label{th:renorm}
Let $X$ a Banach space, then for every $\gamma$-equivalent norm\linebreak $\norma_{L^2(X)}$ there exists a $\gamma$-equivalent norm $\norma_X$ such that
\[ \vartheta_{\norma_X}(t) \geq \overline{\vartheta}_{\norma_{L^2(X)}}(t/2). \]
\end{theo}

We say that an equivalent norm $\eqnorma$ on $L^2(X)$ is \emph{balanced} if 
\[  |I|^{-1} \eqnorm{( f \circ \phi_I ) \chi_I} \leq \eqnorm{f} \]
for every non-trivial interval $I \subset [0,1]$ where $\phi_I: I \rightarrow [0,1]$ the unique affine increasing bijection between those intervals. 

\begin{lema}\label{lemma:balanced}
Given an equivalent norm $\| \cdot \|_1$ on $L^2(X)$, there exists a balanced 
norm $\| \cdot \|_2$ with the same equivalence constants such that
\[ \overline{\vartheta}_{\norma_2}(t) \geq
\overline{\vartheta}_{\norma_1}(t). \] 
\end{lema}

\begin{proof} Define for every non-trivial interval $I \subset [0,1]$ the norm
\[ \|f\|_I = |I|^{-1} \| ( f \circ \phi_I ) \chi_I \|_1 .\]
This norm keeps the same equivalence constants as the map $f \mapsto |I|^{-1} ( f \circ \phi_I ) \chi_I$ is an isometry on its image for the canonical norm on $L^2(X)$. Now, define
$$ \|f\|_2 = \sup\{ \|f\|_I: I \subset [0,1] \} $$
which keeps the same equivalence constants than $\| \cdot \|_1$ and it is easily seen to be a balanced norm.\\
Suppose that 
if $(f_n)_{n=1}^\infty$ is a $t$-separated sequence weakly converging to $f$
with $\|f_n\|_2=1$. Then for any nontrivial $I \subset [0,1]$ we have 
$$ |I|^{-1} \| ( f_n \circ \phi_I ) \chi_I \|_1 = \|f_n\|_I \leq 1$$ 
and $|I|^{-1} ( f_n \circ \phi_I ) \chi_I$ is $t$-separated weakly converging to $|I|^{-1} ( f \circ \phi_I ) \chi_I$. Therefore
\[ \|f\|_I = \| |I|^{-1} ( f \circ \phi_I ) \chi_I \|_1 \leq 1- \overline{\vartheta}_{\| \cdot \|_1}(t). \]
Taking supremum on $I \subset [0,1]$ we get $\|f\|_2 \leq 1- \overline{\vartheta}_{\| \cdot \|_1}(t)$ which clearly implies 
$\overline{\vartheta}_{\norma_2}(t) \geq \overline{\vartheta}_{\norma_1}(t)$ as $(f_n)$ was arbitrary.
\end{proof}

Note that any equivalent norm on $L^2(X)$ can be improved to be stable by measure invariant transformations of $[0,1]$ and keeping the AUC modulus.\\

\begin{proof}[Proof of Theorem~\ref{th:renorm}] After Lemma~\ref{lemma:balanced} we may assume that $\norma_{L^2(X)}$ is balanced. The desired norm $\norma_X$ on $X$ will be its restriction by the canonical inclusion of $X$ into $L^2(X)$. Let $x,y \in B_{\norma_X}$ with $\|x-y\| \geq t$. For each $n\in \mathbb N$, let $f_n$ be the function which takes the values $x$ and $y$ alternatively on the intervals of length $2^{-n}$ of the $n$-dyadic partition of $[0,1]$. As all these functions can be represented as a convex combination of $x$ and $y$ concentrated on intervals, we have $(f_n) \subset B_{\norma_{L^2(X)}}$. Obviously, we have $\|f_n -f_m\| \geq t/2$ and it is easy to see that $(f_n)$ converges weakly to $f(t)=(x+y)/2$ for all $t \in [0,1]$. That implies 
\[\left\| \frac{x+y}{2} \right\|_X = \| f \|_{L^2(X)} \leq 1 -
\overline{\vartheta}_{\norma_{L^2(X)}}(t/2)\]
and the conclusion follows. 
\end{proof}

Indeed, the above argument works also in $L^p(X)$, with $1<p<+\infty$. The previous ideas provide this generalization of a result of Partington \cite{Partington}. 

\begin{coro}
Let $1<p<+\infty$. If $L^p(X)$ has an equivalent balanced AUC norm then the induced norm on $X$ is UC.
\end{coro}

A quantitative version of the result of Partington was proved by Garcia and Johnson in \cite{GJ}: if $L^r(X)$  has an asymptotically uniformly convex renorming of power type $p$, with $1<r\leq p<\infty$, then $X$ has a uniformly convex renorming of power type $p$ if $p=r$, and of power type an arbitrary $p'>p$ when $p>r$. Our argument proves the following improvement that also answers the question in \cite[Remark~4.1]{GJ}.

\begin{coro} Let $1<r, p<+\infty$. Assume that $L^r(X)$ has an asymptotically uniformly convex renorming of power type $p$. Then $X$ has a uniformly convex renorming of power type $p$. 
\end{coro}

\section{Generalized cotypes and renorming}

The definition of generalized cotype $\phi$ given in introduction has the disadvantage of being non-homogeneous in  the sense that the sum $\sum_{k=1}^n \phi(\| x_k \|)$ is not bounded by a function of the averages of the $x_k$'s. In order to remove that obstacle we introduce the following tool.
We say that the cotype $\phi$ has a {\it bound} $\Phi$ if 
\[\sum_{k=1}^n \phi(\norm{x_k})\leq \Phi \left(\int_0^1 \| \sum_{k=1}^n r_k(t) x_k \| \, dt \right) \]
and the function $\Phi$ is convex, non-decreasing and there exists $q \geq 2$ such that $\Phi(t^{1/q})$ is equivalent to a concave function. 
Note that for the classic cotype $p$ we have $\phi(t)=t^p$ and $\Phi(t)=c t^p$.
The existence of cotype bounds was implicitly solved in \cite{figiel2} where Figiel actually showed that 
$\Phi(t)$ can be taken of the form $c(t^2+t^p)$ where $c>0$ and $p \geq 2$. However we will not use an explicit form for the bound because the results are more general that way and the proofs clearer. We left to the reader the easy task of checking 
that such a cotype bound $\Phi$ satisfies the $\Delta_2$ condition. 

\begin{prop}\label{bound}
Let $\phi$ a cotype of $X$ with bound $\Phi$. Then there is $q \geq 2$ and $C>0$ such that 
$$ \int_0^1 \sum_{k=1}^n \phi(\|f_k(s)\|) \, ds \leq C \, \Phi \left( \int_0^1 \| \sum_{k=1}^n r_k(t) f_k \|_{L^q(X)} \, dt \right)$$
whenever $f_1,\dots,f_n \in L^q(X)$.
\end{prop}

\begin{proof}
Let $q \geq 2$ given by the definition of cotype bound and $\eta$ a concave function equivalent to $\Phi(t^{1/q})$.
According to Kahane inequalities and the $\Delta_2$ property there is $a>0$ such that
\begin{align*}
 \int_0^1 \sum_{k=1}^n \phi(\|f_k(s)\|) ds &\leq a \int_0^1 \Phi \left( (\int_0^1 \|  \sum_{k=1}^n r_k(t)f_k(s) \|^q dt)^{1/q} \right) ds \\
& \leq a b \int_0^1 \hspace{-1mm} \eta  \left( \int_0^1 \|  \sum_{k=1}^n r_k(t)f_k(s) \|^q dt \right) ds\\
& \leq a b \, \eta  \left( \int_0^1 \hspace{-2mm}\int_0^1 \|  \sum_{k=1}^n r_k(t)f_k(s) \|^q dt \,  ds \right) \\
&  \leq ab \, \eta\left( \int_0^1 \| \sum_{k=1}^n r_k(t)f_k \|^q_{L^q(X)} \, dt \right) \\
&\leq abc \, \Phi\left( \int_0^1 \| \sum_{k=1}^n r_k(t) f_k \|_{L^q(X)} \, dt \right) \end{align*}
where $b,c>0$ are the constants of equivalence between  $\Phi(t^{1/q})$ and $\eta(t)$.
\end{proof}

\begin{coro}\label{uncon_series}
Let $\phi$ a cotype of $X$. Then there exists a $q \in [2,+\infty)$ such that
$$ \sum_{n=1}^\infty \int_0^1  \phi(\|f_n(s)\|) \, ds < +\infty $$
whenever the series $\sum_{n=1}^\infty f_n$ is unconditionally convergent in $L^q(X)$.
\end{coro}

Next result says that if $\phi$ is a cotype of $X$, then it ``almost'' is a cotype of $L^2(X)$.

\begin{prop}
Let $\phi$ a cotype of $X$. Then there is a  cotype bound $\Phi$,  $q \geq 2$ and $C>0$ such that
$$ \sum_{k=1}^n \phi(\| f_k \|_{L^2(X)}) \leq C \, \Phi \left(\int_0^1 \| \sum_{k=1}^n r_k(t) f_k \|_{L^q(X)} dt \right) $$
whenever $f_1,\dots,f_n \in L^q(X)$.
\end{prop}

\begin{proof}
According to Figiel \cite[Theorem~1.8]{figiel2} $\phi$ can be bounded above by a cotype $\overline{\phi}$ such that $ \xi(t)=\overline{\phi}(t^{1/2})$ is  convex. Therefore
$$  \sum_{k=1}^n \phi(\| f_k \|_{L^2(X)}) \leq  \sum_{k=1}^n \overline{\phi}(\| f_k \|_{L^2(X)}) ~~~~~~~~~~~~~~~~~~~~ $$
$$ \leq   \sum_{k=1}^n  \xi(\int_0^1 \| f_k(s) \|^2 ds )
\leq    \sum_{k=1}^n \int_0^1 \xi( \| f_k (s)\|^2 ) ds ~~~~~~~~~~~~ $$
$$ =  \sum_{k=1}^n \int_0^1 \overline{\phi}(\| f_k(s) \|) ds  \leq C \, \Phi \left(\int_0^1 \| \sum_{k=1}^n r_k(t) f_k \|_{L^q(X)} dt \right)  $$
where the last inequality comes from Proposition~\ref{bound} being $\Phi$ a bound for $\overline{\phi}$.
\end{proof}

We do not know if a cotype of $X$ is always a cotype of $L^2(X)$. As a consequence of renorming results we will give partial results in this sense, Corollary~\ref{hereda-cotype}. We also have the following well-known result.

\begin{coro}
If $X$ has cotype $q$ then so does $L^q(X)$ for $q \geq 2$.
\end{coro}

\begin{proof}
In that case we may take the same $q$ associated with the cotype bound $t^q$, so (the proof of) the previous proposition will lead to the result.
\end{proof}

Now we will consider Banach space valued martingales. For our purposes it is enough to consider  \emph{Walsh-Payley martingales}  defined on $[0,1]$, that is, functions are measurable with respect to the dyadic partitions of $[0,1]$.
The set of all $B_X$-valued Walsh-Payley martingales is denoted $\M(B_X)$. 
Given a martingale $(f_n)_{n=0}^\infty \in \M(B_X)$, we denote $df_n := f_{n}-f_{n-1}$, with the convention that $df_0 = f_0$.\\

Next result is part of the characterizations obtained by Garling in \cite[Theorem~3]{Garling}, build upon the seminal work of Pisier 
	\cite{Pisier75}. It provides a criterium for the existence of an equivalent norm whose modulus of uniform convexity is better than a given function. 

\begin{theo}[\cite{Garling} Theorem~3]\label{prop:cotr}
Let $\phi(t) \geq 0$ an Orlicz function with the $\Delta_2$ property. Then the Banach space $X$ has an equivalent norm 
$\eqnorma{}$ with $\delta_{\eqnorma{}}(t) \geq c \, \phi(t) $ for some $c>0$ if and only if there exists a constant $C>0$ such that 
$$\sum_{n=0}^\infty \int_0^1 \phi(\norm{df_n(t)})dt < +\infty $$ 
whenever $(f_n)_{n=0}^\infty\in \M(B_X)$. 
\end{theo}
	
\begin{proof}[Proof of Theorem~\ref{UMD-renorm}]
According to Figiel's improvement of cotypes \cite[Theorem~1.8]{figiel2} we may assume that $t\mapsto \phi(t^{1/2})$ is convex and satisfies the $\Delta_2$ property, so $\phi$ is convex and $t^{-1}\phi(t)\to 0$ as $t\to 0$.
Let $q\geq 2$ the number given by Corollary~\ref{uncon_series}. Any martingale $(f_n)_{n=0}^\infty \in \M(B_X)$ is bounded in $L^q(X)$. Since the Banach space is UMD, the series $\sum_{n=1}^\infty df_n$ is unconditionally convergent and therefore the hypothesis of Theorem~\ref{prop:cotr} is fulfilled by the thesis of Corollary~\ref{uncon_series}.
 \end{proof}

\begin{rema}
Note that we do not need all the power of the definition of UMD space, that is, the uniform boundedness of all the partial sums of the series with arbitrary changes of signs. It is enough the boundedness of the averaged sums 
$$ \sup_n \int_0^1 \| \sum_{k=1}^n r_k(t) df_k \|_{L^q(X)} < +\infty $$
for any martingale $(f_n)$ bounded in $L^q(X)$.
\end{rema}

We finish with this straightforward application.

\begin{coro}\label{hereda-cotype}
Let $X$ be an UMD space or a super-reflexive space with l.u.st. (in particular if $X$ has an unconditional basis). Then any cotype of $X$ is a cotype of $L^2(X)$ too.
\end{coro}

\section{Best modulus and best cotype}

In \cite{raja2} the second-named author proved that the modulus of convexity cannot be improved (from an asymptotic point of view) if and only if it is equivalent to  $\NN_X(t)^{-1}$. In such a case, $\NN_X(t)^{-1}$ would be as well a cotype of $X$ (the second-named author regrets to say that Proposition 5.5 in
\cite{raja2}, which says that $\NN_X(t)^{-1}$ is always a generalized cotype of $X$, is not correct). 
The renorming results with cotype function, as Theorem~\ref{UMD-renorm}, would imply that $\NN_X(t)^{-1}$ is the best cotype too. In general, the best modulus of convexity or the best cotype is not attainable. Figiel built a super-reflexive space $F$ with symmetric basis not having a renorming with best modulus of convexity neither a best cotype. Note that the space $L^2(F)$ provides a counter-example to the existence of renormings with best AUC modulus after Theorem~\ref{th:renorm}.

\begin{proof}[Proof of Theorem~\ref{sup-cotype}]
The construction of $\QQ$ is as follows. Firstly we claim that for every $\varepsilon>0$ there is a least $\QQ(\varepsilon)$ such that if $n \geq \QQ(\varepsilon)$ and $x_1,\dots,x_n \in X$ are such that

$$  \int_0^1 \| \sum_{k=1}^n r_k(t) x_k \| \, dt \leq 1 $$
then $n^{-1} \sum_{k=1}^n \| x_k \| \leq \varepsilon$. That can be deduced easily from the fact that $X$ has 
nontrivial Rademacher classic cotype $p \geq 2$. Nevertheless we will present a direct argument just using the very definition of super-reflexive space. Together with submultiplicativity of $\QQ$ proved below, that would provide a proof of the existence of nontrivial cotypes in super-reflexive Banach spaces not resorting to uniformly convex renorming.\\
Suppose that the statement is false, so for some $\varepsilon>0$ one can find arbitrarily long sequences $x_1,\dots,x_n$ such that $  \int_0^1 \| \sum_{k=1}^n r_k(t) x_k \| \, dt \leq 1 $ and $n^{-1} \sum_{k=1}^n \| x_k \| > \varepsilon$. Those finite sequences can be mixed in one infinite sequence $(x_n)_{n=1}^\infty$ in the ultra-power $X^{\mathcal U}$ where ${\mathcal U}$ is a free ultrafilter over ${\Bbb N}$. The sequence preserve the properties in the following way $  \int_0^1 \| \sum_{k=1}^n r_k(t) x_k \| \, dt \leq 1 $ and 
$n^{-1} \sum_{k=1}^n \| x_k \| \geq \varepsilon$ for every $n \in {\Bbb N}$. Put $f_n(t)=r_n(t)x_n$ and consider the sequence $(f_n)_{n=1}^\infty$ in the reflexive space $L^2(X^{\mathcal U})$. It is well known that 
$(f_n)_{n=1}^\infty$ is a monotone basic sequence  \cite[p. 78]{LT} and $\sup_n \| \sum_{k=1}^n f_k \|_{L^2(X^{\mathcal U})} <\infty$ thanks to Kahane inequality.
On the other hand, we have  $\lim_n n^{-1} \sum_{k=1}^n \|f_k\|_{L^2(X^{\mathcal U})} \not =0$ which means that 
$(f_n)_{n=1}^\infty$ is not boundedly complete \cite[Theorem 4.15]{banach}. That denies the reflexivity of $L^2(X^{\mathcal U})$ proving the existence of $ \QQ(\varepsilon)$.\\
Now we will prove the submultiplicativity of $\QQ$. Indeed, take $\varepsilon_1, \varepsilon_2>0$, $n=\QQ(\varepsilon_1)$ and $m=\QQ(\varepsilon_2)$. Let $x_1,\dots,x_{nm} \in X$ such that 
$$  \int_0^1 \| \sum_{k=1}^{nm} r_k(t) x_k \| \, dt \leq 1 .$$
Take $\overline{r}_j=r_{nm+j}$ the Rademacher functions in order to have more of them with simple indices. For each $t \in [0,1]$ we have
$$ m^{-1} \sum_{k=1}^m \| \sum_{j=1}^n r_{n(k-1)+j}(t) x_{n(k-1)+j} \| $$
$$ \leq \varepsilon_2 \int_0^1 \| \sum_{k=1}^m \overline{r}_k(s) \sum_{j=1}^n r_{n(k-1)+j}(t) x_{n(k-1)+j} \| ds.  $$
When integrating with respect to $t$ the last inequality note that 
$$  \int_0^1 \int_0^1 \| \sum_{k=1}^m  \sum_{j=1}^n \overline{r}_k(s) r_{n(k-1)+j}(t) x_{n(k-1)+j} \| ds \, dt 
=  \int_0^1 \| \sum_{k=1}^{nm} r_k(t) x_k \| \, dt $$
because any choice of signs in the sequence $\{ \pm x_1,\dots, \pm x_n\}$ produced by $\overline{r}_k(s)$ will lead to the same Rademacher average.
Therefore we have
$$ m^{-1} \sum_{k=1}^m \int_0^1 \| \sum_{j=1}^n r_{n(k-1)+j}(t) x_{n(k-1)+j} \| dt \leq \varepsilon_2 .$$ 
On the other hand
$$ n^{-1} \sum_{j=1}^n \| x_{n(k-1)+j} \| \leq \varepsilon_1 \int_0^1 \| \sum_{j=1}^n r_{n(k-1)+j}(t) x_{n(k-1)+j} \| dt $$
Multiplying by $m^{-1}$ and summing up with respect to $k$ we finally get
$$ n^{-1}m^{-1} \sum_{j=1}^{nm} \| x_j\| \leq \varepsilon_1\varepsilon_2 $$
which implies that $nm \geq \QQ(\varepsilon_1 \varepsilon_2)$.\\
Let $\phi(t)$ be a convex normalized cotype and fix $t \in (0,1)$. If $n=\QQ(t)-1$ there there exist $x_1,\dots,x_n \in X$ such that 
$$  \int_0^1 \| \sum_{k=1}^n r_k(t) x_k \| \, dt \leq 1 $$
and $n^{-1} \sum_{k=1}^\infty \| x_k \| > t$. Therefore
$$ \phi(t) \leq n^{-1} \sum_{k=1}^{n} \phi(\|x_k\|) \leq n^{-1} = (\QQ(t)-1)^{-1}.$$
We will prove now that for every $\varepsilon>0$ the function $\phi_\varepsilon(t)$ is a cotype. If $x_1,\dots,x_n$ are such that 
$$  \int_0^1 \| \sum_{k=1}^n r_k(t) x_k \| \, dt \leq 1 $$
then 
$$ \# \{ k: \|x_k\| \geq \varepsilon \} \leq \QQ(\varepsilon) $$
which implies $\sum_{k=1}^n \phi_\varepsilon(\|x_k\|) \leq 1.$ as desired.
Finally, according to \cite[Theorem~1.8]{figiel2} an arbitrary cotype $\phi(t)$ is bounded above by a convex cotype, and therefore it is bounded above by $c\, \QQ(t)^{-1}$ for a suitable constant $c>0$.
\end{proof}

\begin{rema}
Note that if we could use $\min_{1\leq k \leq n} \|x_k\|$ instead of $n^{-1} \sum_{k=1}^n \|x_k\|$ in the proof above we could remove the hypothesis of convexity for the the cotype $\phi$. However, the corresponding definition of $\QQ$ does not seem to be submultiplicative. Compare with submultiplicative property of some constant associate to monotone basic sequences in super-reflexive spaces, as a part of the proof of  \cite[Theorem~10.9]{Pisier_libro}. 
\end{rema}

If there is a way to determine $\QQ_X$ then we have explicit almost optimal generalized cotypes.

\begin{prop}\label{prop:O_X}
Let $X$ be a super-reflexive space. Then for every $\alpha>1$ the following function
$$ \phi_\alpha(t) =  |\log 1/t |^{-\alpha} \QQ_X(t)^{-1} $$
is a generalized cotype of $X$.
\end{prop}

\begin{proof}
Let $n_0$ be such that $\phi_\alpha(t)$ is increasing on $[0,2^{-n_0}]$ (the existence comes from the fact that $\QQ_X(t)^{-1} \preceq t^2$) and let $M$ the maximum of $\phi_\alpha(t)$ on $[0,1/2]$. 
Let $x_1,\dots,x_m$ such that $\int_0^1 \| \sum_{k=1}^m r_k(t) x_k \| \leq 1/2$. 
For every $n \geq n_0$ we have
$$ \#\{ k : 2^{-n-1} < \|x_k\| \leq 2^{-n}  \}  \leq \QQ_X(2^{-n-1}) \leq C \, \QQ_X(2^{-n})$$
where $C=\QQ_X(1/2)^{-1}$. Now
\begin{align*}
\sum_{2^{-n-1} < \|x_k\| \leq 2^{-n}} \phi_\alpha(\|x_k\| ) &\leq C  (n\log 2 )^{-\alpha} \QQ_X(2^{-n})^{-1} \QQ_X(2^{-n})\\
& = C (\log 2 )^{-\alpha} n^{-\alpha} 
\end{align*} 
and therefore
\[ \sum_{k=1}^m \phi_\alpha(\|x_k\|) \leq C (\log 2 )^{-\alpha} \sum_{n=1}^\infty n^{-\alpha} + M \, \QQ(2^{-n_0}) .\]
As the constant on the right-hand side does not depend on the choice of the $x_k$'s we have proved that $\phi_\alpha(t)$ is a cotype defined on $[0,1/2]$.
\end{proof}

\begin{coro} Given a super-reflexive Banach space $X$, consider 
	\[\mathfrak{p}_X = \inf_{0<t<1}\frac{\log(\mathfrak{O}_X(t))}{\log(1/t)}.\]
Then every $p>\mathfrak{p}_X$ is a a classic Rademacher cotype of $X$. 	
\end{coro}

\begin{proof} Take $p>p'>\mathfrak{p}_X$ and follow the same steps as in the proof of Corollary 1.3 in \cite{raja2} to show that $\mathfrak{O}_X\preceq t^{-p'}$. This implies that $t^p \preceq \phi_\alpha$ for any $\alpha >1$. Proposition~\ref{prop:O_X} finishes the proof.  
\end{proof}

There is always a ``best cotype'' if we think of sequence spaces rather than of functions. The {\it cotype space} was introduced by Figiel \cite{figiel2} as follows. Given a Banach space $X$ we will consider the set $A$ of all the $(t_i) \in c_{00}$ such that there are $x_i \in X$ such that $t_i=\|x_i\|$ and

$$  \int_0^1 \| \sum_{k=1}^n r_k(t) x_k \| \, dt \leq 1 .$$

Now let $B$ the convex hull of $A$ and complete $\mbox{span}(B)$ to be Banach space $E$ so as $B$ is dense in $B_E$. The space $E$, which has a symmetric basis by construction, is called the cotype space of $X$. Note that an Orlicz function $\Phi$ is a cotype of $X$ if and only if the cotype space $E$ of $X$ imbeds into $\ell_\Phi$.\\

We will turn now our attention again to the function $\NN_X$.
The next two results show how $ \NN_X(t)^{-1}$ is close in properties to a modulus of convexity in spite of there could not be a best modulus of convexity.

\begin{prop}\label{best-st}
Let $X$ be a super-reflexive Banach space and let $\Phi(t) = \NN_X(t)^{-1}$. Then
\begin{itemize}
\item[(a)] $\Phi(t)$ is equivalent to a supermultiplicative function;
\item[(b)] $\Phi(t^{1/2})$ is equivalent to a convex function.
\end{itemize}
\end{prop}

\begin{proof}
Statement (a) is a consequence of the known submultiplicativity of $\NN_X(t)$ with the original definition (see Theorem~\ref{main1}) or the equivalent function $\Sz(B_{L^2(X)},t)$. Statement (b) comes from the fact that for every equivalent norm $t^{-2}\delta_\eqnorma (t)$ is equivalent to a non-increasing function with a universal constant of equivalency (just follow the constants along the proofs in \cite[pp. 65--66]{LT}).
\end{proof}

\begin{prop}
Let $\Phi(t)$ a function satisfying the statements (a) and (b) of Proposition~\ref{best-st}. Then there exists a uniformly convex Banach space $X$ such that $\delta_X(t) \sim \Phi(t)$ and that modulus cannot be improved asymptotically (equivalently, $\Phi(t)^{-1} \sim \NN_X(t)$).
\end{prop}

\begin{proof}
The function $\Phi$ is equivalent to a convex function $\phi$ which satisfies the $\Delta_2$ property at $0$. Then
the Orlicz space $\ell_\phi$ does the work after the fine estimations in \cite{MT}, see also \cite[p. 67]{LT}.
\end{proof}

We include the following sufficient condition for the optimality of a modulus of uniform convexity based on the modulus of asymptotic uniform smoothness $\overline{\rho}_X(t)$ which is defined as
\[ \overline{\rho}_{\norma}(t) = \sup_{\norm{x}=1}\inf_{\dim(X/Y)<\infty}\sup_{y\in Y, \norm{y}=1} \norm{x+ty}-1 .\]

\begin{prop}
Let $X$ be a uniformly convex space such that $\overline{\rho}_{\eqnorma}(t) \preceq \delta_X(ct)$ for some $c>0$ and some equivalent norm $\eqnorma$. Then the modulus of uniform convexity of $X$ is already a best one.
\end{prop}

\begin{proof}
By \cite[Proposition~4.9]{raja3} we have $\Sz(B_X,t)^{-1} \leq a \, \overline{\rho}_{\eqnorma}( b t) $ for some $a,b >0$. The hypothesis implies now that  $\NN_X(t)^{-1} \leq a' \delta_X(b' t)$ for some $a',b'>0$. Playing with the submultiplicativity of $\NN_X$ we get easily that $\NN_X(t)^{-1} \preceq \delta_X(t)$  which is equivalent to the desired conclusion.
\end{proof}

Note that  $\overline{\rho}_{\ell^p}(t) \sim t^p $ for $1 \leq p < +\infty$, so the optimality of the modulus of uniform convexity of $\ell^p$ follows from the result for $p \geq 2$ (the case $1<p<2$ is obvious because the modulus is already $\sim t^2$). As $L^p$ contains an isometric copy of $\ell^p$ the optimality of the modulus of uniform convexity of $L^p$  follows as a consequence.\\ 

Finally, we will prove the result announced in the introduction on the existence of a uniformly convex  renorming with modulus of power type $2$.

\begin{proof} {\it of Theorem~\ref{dos}}.
The hypothesis and  \cite[Theorem~1.5]{raja2} implies that for some $a>0$ such that for every $\tau \in (0,1]$ there is a 2-equivalent norm $\norma_\tau$ such that 
$$ \delta_{\norma_\tau}(\tau) \geq a \tau^2 .$$
On the other hand $t^{-2}  \delta_{\norma_\tau}(t)$ is equivalent to a non-decreasing function with a constant $b>0$ which is universal \cite[Corollary~11]{figiel} thus
$$ b \, \frac{\delta_{\norma_\tau}(\tau)}{\tau^2} \leq   \frac{\delta_{\norma_\tau}(t)}{t^2} $$
if $t \geq \tau$. Therefore $\delta_{\norma_\tau}(t) \geq ab \, t^2$ if $t \geq \tau$. Now, let $\norma_n$ the norm $\norma_\tau$ for $\tau=1/n$ and define
$$ \eqnorm{ x } = \lim_n \, \sup\{ \|x\|_m: m \geq n \}. $$
This norm verifies that $\delta_{\eqnorma}(t) \geq ab \, t^2$ just arguing like in the proof of Proposition~\ref{glue}. Clearly the modulus is not asymptotically improvable by N\"ordlander's \cite[Proposition~A.1]{BL} and neither the cotype by Dvoretzky's \cite[Theorem~12.10]{BL}.
\end{proof}

\bibliographystyle{siam}
\bibliography{biblio}

\end{document}